\documentclass[12pt]{article}%
\usepackage{amsfonts, amsmath, amssymb,amscd}
\usepackage[all]{xy}
\usepackage{latexsym}
\usepackage{amsmath, amsthm}
\usepackage{amsfonts}
\usepackage{amssymb}
\usepackage{amsmath}
\usepackage{graphicx,color}
\usepackage{comment}
\usepackage{tikz}
\usepackage{verbatim}
\usepackage[normalem]{ulem}%
\usepackage{graphicx}
\providecommand{\U}[1]{\protect\rule{.1in}{.1in}}
\DeclareMathAlphabet{\mathpzc}{OT1}{pzc}{m}{it}
\newtheorem{theorem}{Theorem}[section]
\newtheorem{lemma}[theorem]{Lemma}
\newtheorem{proposition}[theorem]{Proposition}
\newtheorem{definition}[theorem]{Definition}
\newtheorem{corollary}[theorem]{Corollary}
\newtheorem{remark}[theorem]{Remark}

\newcommand{\vetor}[1]{\mathbf{#1}}

\newcommand{\capa}{{\rm \; cap}_p \,}
\newcommand{\Capa}{{\rm \; Cap}_p \,}

\newcommand{\osc}[1]{\mathop{\rm osc}_{#1}}

\newcommand\BBLfootnote[1]{%
  \begingroup
  \renewcommand\thefootnote{}\footnote{#1}%
  \addtocounter{footnote}{-1}%
  \endgroup
}

\begin{document}

\title{Equivalences among parabolicity,  comparison principle and capacity on complete
Riemannian manifolds }
 \author{\BBLfootnote{This work was partially supported by the ``Brazilian-French Network in Mathematics", a project sponsored by CNPq, from Brazil, and by Fondation des Sciences Math\'ematiques de Paris (FSMP), from France.} A. Aiolfi, L. Bonorino, J. Ripoll, M. Soret, and M. Ville}
\date{}
\maketitle

\begin{abstract}
In this work we establish new equivalences for the concept of $p$-parabolic Riemannian manifolds. We define a concept of comparison principle for elliptic PDE's on exterior domains of a complete Riemannian manifold $M$ and prove that $M$ is $p$-parabolic  if and only if this comparison principle holds for the $p$-Laplace equation. We show also that the $p$-parabolicity of $M$ implies the validity of this principle for more general elliptic PDS's and, in some cases, these results can be extended for non $p$-parabolic manifolds or unbounded solutions, provided that some growth of these solutions are assumed. 

\end{abstract}

\section{Introduction}

\qquad We recall that a complete Riemannian manifold $M$ is $p-$parabolic,
 $p~>~1,$ if $M$ does not admit a Green function that is, a positive real
function $G\left(  x,y\right)  $ defined for $x,y\in M$ with $x\neq y$  such that  $G\left(
\cdot,y\right)  $ is of $C^{1}$ class in $M\backslash\left\{  y\right\}  $ for
any fixed $y\in M,$ and%
\[
-\operatorname{div}\left(  \left\Vert \nabla G\left(  \cdot,y\right)
\right\Vert ^{p-2}\nabla G\left(  \cdot,y\right)  \right)  =\delta_{y},\text{
}y\in M.
\]
This equality is in the sense of distributions that is,%
\[
\int_{M}\left\langle \left\Vert \nabla G\left(  \cdot,y\right)  \right\Vert
^{p-2}\nabla G\left(  \cdot,y\right)  ,\nabla\varphi\right\rangle
=\varphi\left(  y\right)
\]
for all $\varphi\in C_{c}^{1}\left(  M\right)  .$

A well known subject of research is to find characterizations for the
$p-$parabolicity of complete Riemannian manifolds. Gathering together results
which have been proved in the last decades, see Grigor'yan \cite{Gr0} and Holopainen \cite{H0, H1, H1-5, H2} for instance, we have
 (the notions used in the
theorem are all defined below):

\begin{theorem}
\label{gather}Let $M$ be a complete Riemannian manifold. Then the following
alternatives are equivalent:

\qquad(a) $M$ is $p-$parabolic;

\qquad(b) the $p-$capacity of any compact subset of $M$ is zero;

\qquad(c) the $p-$capacity of some precompact open subset of $M$ is zero;

\qquad(d) any bounded from below supersolution for the $p-$Laplacian operator is constant.

\end{theorem}

The case $p=2$ is more classical, with a longer history, and there are
other equivalences (see \cite{Gr0,Gr}).

For a domain $\Omega\subset M$, we define the
$p$-capacity of a set $E\subset\Omega$ with respect to $\Omega$ 
by
\begin{equation}
\mathrm{\;cap}_{p}\,(E;\Omega):=\inf_{v\in\mathcal{F}_{E,\Omega}}\int
_{\Omega}|\nabla v|^{p}dx,\label{p-capacity}%
\end{equation}
where the infimum is taken over all the functions $v$ that belongs to 
\begin{equation}
 \mathcal{F}_{E,\Omega} = \{ v \in C^{\infty}_c(M) \, : \, v = 1 \; \; {\rm in } \; \; E \; \; {\rm and} \; \; {\rm supp}(v) \subset \Omega \}.
\label{AdmissibleLocally} 
\end{equation}

Observe that the $p-$capacity of a bounded subset $E\subset M$ with respect to $M$ can be obtained by taking a sequence of geodesic balls $B(o,R_{k})$   centered at a given point $o\in M,$ with $R_{k} \to+\infty$,  and taking the limit
\begin{equation}
\mathrm{\; Cap}_{p} \, (E) := \lim_{k \to+\infty} \mathrm{\; cap}_{p} \, (E;
B(o,R_{k})). \label{CapaIsLimitOfcapa}%
\end{equation}
It is not difficult to prove that the limit always exists, that does depend not  on the sequence of  balls and that 
$\mathrm{\; Cap}_{p} \, (E)=
\mathrm{\; cap}_{p} \, (E;M).$

Now we observe that as a result of the equivalence between $(b)$ and $(c)$ of Theorem \ref{gather} (see \cite{H0, H1, H1-5, H2}),  we obtain the following result which is an important tool to the study of $p$-parabolicity and comparison principle: 
\begin{corollary}
[Dichotomy]Let $M$ be a complete noncompact Riemannian manifold and $p>1$.
If$\mathrm{\;Cap}_{p}\,(B_{0})>0$ for some ball $B_{0}\subset M$, then
\[
\mathrm{\;Cap}_{p}\,(B)>0
\]
for any ball $B\subset M$. More generally, if$\mathrm{\;Cap}_{p}\,(E)>0$ for
some bounded set $E$, then $\mathrm{\;Cap}_{p}\,(F)>0$ for any set $F$ that
contains some interior point. \label{Dichotomy}
\end{corollary}

Recall that  the $p-$Laplacian operator $\Delta
_{p}$ is defined by%
\[
\Delta_{p}u=\operatorname{div}\left(  \left\Vert \nabla u\right\Vert
^{p-2}\nabla u\right)  ,\text{ }u\in C^{1}\left(  M\right).
\]
A function  $w\in C^{1}\left(  M\right)  $ is a supersolution of $\Delta_{p}$ if
$\Delta_{p}w\leq0$ in the weak sense.

In the main result of our paper we add two other equivalences in Theorem
\ref{gather}. The main one relates $p-$parabolicity with another very classical
notion, the comparison principle, which is likely the most fundamental tool on PDE (see \cite{PS}, Abstract). In the second one we use the $p-$capacity to show in an explict way that the $p-$parabolicity of a Riemannian manifold depends on its behaviour at infinity. 

We first define that $M$ satisfies the comparison principle for the $p-$Laplace
operator (or the $p-$comparison principle for short) if, given any exterior domain
$\Omega$ of $M$ ($\Omega = M \backslash K$, where $K$ is a compact subset), if $v,w\in
C^{0}\left(  \overline{\Omega}\right)  $ are bounded sub and super solutions for
$\Delta_{p}$ then, whenever $v|\partial\Omega\leq w|\partial\Omega,$ it
follows that $v\leq w$ in $\Omega.$ (In this case, we also say that the comparison principle for exterior domains in $M$ holds for the $p$-Laplace operator.)
We prove:

\begin{theorem}
\label{MT}Let $M$ be a complete Riemannian manifold. Then the following alternatives are equivalent:

\qquad (a)  $M$ is $p-$parabolic

\qquad (b)  $M$ satisfies the comparison principle for $\Delta_{p}$

\qquad (c) Given $o\in M$ there are sequences $R_1^k, R_2^k$ satisfying 
$R_{1}^{k}<R_{2}^{k}$ and%
\[
\lim_{k\rightarrow\infty}R_{1}^{k}=\lim_{k\rightarrow\infty}R_{2}^{k}=\infty
\]
\qquad \qquad such that
\[
\sup_{k} \mathrm{\; cap}_{p} \,
(B(o,R_{1}^k);B(o,R_{2}^k)) < +\infty\]
\end{theorem}

 The existence of many different conditions implying
$p-$parabolicity is well known and hence, where the $p-$comparison principle holds (see \cite{HP}, \cite{HMP}, \cite{T}, \cite{EP}, \cite{H1},
\cite{LT}, \cite{S}, \cite{H2} and also references therein). Among them,
there are several works giving conditions in terms of the volume growth of
geodesic balls. We prove here the following result:

\begin{theorem}
\label{PGGB}Let $M$ be a complete noncompact Riemannian manifold and $p>1$.
Suppose that, for a fixed point $o\in M$, there exist two increasing sequences
$r_{k}$ and $s_{k}$ such that $s_{k}>r_{k}\rightarrow+\infty$
and
\begin{equation}
\sup_{k}\left(  \frac{2}{s_{k}-r_{k}}\right)  ^{p}\operatorname*{Vol}%
(B(o,s_{k})\backslash B(o,r_{k}))<+\infty.\label{UpperBoundBallGrowth}%
\end{equation}
Then $M$ is $p-$parabolic.
\end{theorem}

We note that condition \eqref{UpperBoundBallGrowth} is satisfied if there
exists some sequence $r_{k}\rightarrow+\infty$ such that
\[
\operatorname*{Vol}(B(o,r_{k}))\leq Cr_{k}^{q}\quad\mathrm{for\;any}%
\;k,
\]
where $q\leq p$ and $C>0$ is a constant that depends only on $M$.

\section{A general comparison result}

In this section, we prove that a comparison result holds for a  general class of equations which contains the $p-$Laplacian provided that the $p-$capacity of a sequence of annuli in the manifold has some decay. This section is of independent interest, but it will be important for the next one.  

We will study a class of second order elliptic PDE's of the form

\begin{equation}
{\rm div} \left( \frac{ A (|\nabla u|) }{|\nabla u|}  \nabla u \right) = 0 \quad {\rm in} \quad M \backslash K,
\label{exteriorProblem}
\end{equation}
where $K \subset M$ is a compact set and $A \in C[0,+\infty) \cap C^1(0,+\infty)$ satisfies, for some $\alpha \in (0,1]$ and $\beta \ge 1$, the following conditions:
\begin{align}
 \bullet & \; A(0)=0 \text{ and }  A(s) > 0 \; \text{for } s > 0;  \label{AdeZeroeZero} \\
 \bullet & \; s A'(s) \ge \alpha A(s)  \; \text{for } s \ge 0; \label{AprimeHasSomeGrowth} \\
 \bullet & \; s A'(s) \le \beta A(s)  \; \text{for } s \ge 0; \label{AprimeHasSomeGrowthBounded} \\
  & \text{There exist }D_2  \text{ and } D_1  \; \text{positive such that} \nonumber \\
 \bullet & \; A(s) \ge D_1 \, s^{p-1} \; \text{for any } s \ge 0; \label{LowerBoundForA} \\
 \bullet & \; A(s) \le D_2 \, s^{p-1} \; \text{for any } s \ge 0. \label{UpperBoundForA}
\end{align}

Related to the operator in \eqref{exteriorProblem}, define $\mathcal{A}:\mathbb{R}^n \to \mathbb{R}^n$ by 
$$ \mathcal{A}(0) = 0 \quad \text{and} \quad \mathcal{A}(\vetor{v})= \frac{A(|\vetor{v}|)}{|\vetor{v}|} \vetor{v} \quad \text{for } \; \vetor{v} \ne 0.$$
Hence equation \eqref{exteriorProblem} is equivalent to 
$$ {\rm div} \mathcal{A}(\nabla u ) = 0 \quad \text{in} \quad M\backslash K.$$
Remind that $u \in C^1(M \backslash K)$ is said to be a weak solution to this equation if 
\begin{equation*}
  \int_{M \backslash K} \nabla \varphi \cdot \mathcal{A}(\nabla u ) \, dx = 0
\end{equation*} 
for any $\varphi \in C_c^{1} (M \backslash K)$. By an approximating argument, this holds for any $\varphi \in W_0^{1,p} (M \backslash K)$.

Observe that conditions \eqref{AprimeHasSomeGrowth} and \eqref{LowerBoundForA}
imply that
\begin{equation}
 \sum_{i,j=1}^{n} \frac{\partial \mathcal{A}_j(\vetor{v})}{\partial \vetor{v}_i} \xi_i \xi_j \ge  D_1 \alpha|\vetor{v}|^{p-2} |\xi|^2  
\end{equation}
and, from \eqref{AprimeHasSomeGrowthBounded} and \eqref{UpperBoundForA}, we have that 
\begin{equation}
 \sum_{i,j=1}^{n} \left| \frac{\partial \mathcal{A}_j(\vetor{v})}{\partial \vetor{v}_i} \right| \le  n^2 D_2  (\beta + 2) |\vetor{v}|^{p-2} 
\end{equation}
for any $\xi \in \mathbb{R}^n$ and $\vetor{v}  \in \mathbb{R}^n \backslash \{0\}$, where $\mathcal{A}_j$ is the $j$th coordinate functions  of $\mathcal{A}$.
Moreover, $\mathcal{A} \in C(\mathbb{R}^n ; \mathbb{R}^n) \cap C^1(\mathbb{R}^n \backslash \{0\}; \mathbb{R}^n)$ and $\mathcal{A}(0)=0$.
Hence, we have the following result, proved in \cite{D} (Lemma 2.1):

\begin{lemma}
There exist positive constants $c_1 = c_1(n,p,\alpha)$ and $c_2=c_2(n,p,\beta)$ such that
\begin{align} 
(\mathcal{A}(\vetor{v}_2) - \mathcal{A}(\vetor{v}_1)) \cdot (\vetor{v}_2 - \vetor{v}_1) &\ge c_1 (|\vetor{v}_2| + |\vetor{v}_1|)^{p-2} |\vetor{v}_2 - \vetor{v}_1|^2 
\label{GenericInequalityFromBelowForA}
\\[5pt] 
 |\mathcal{A}(\vetor{v}_2) - \mathcal{A}(\vetor{v}_1) | &\le c_2 (|\vetor{v}_2| + |\vetor{v}_1|)^{p-2} |\vetor{v}_2 - \vetor{v}_1| 
\label{GenericInequalityFromAboveForA}
\end{align}
for $\vetor{v}_1, \vetor{v}_2 \in \mathbb{R}^n$ that satisfy $|\vetor{v}_1 | + |\vetor{v}_2| \ne 0$. If $p \ge 2$, then
\begin{equation} 
(\mathcal{A}(\vetor{v}_2) - \mathcal{A}(\vetor{v}_1)) \cdot (\vetor{v}_2 - \vetor{v}_1) \ge c_1 |\vetor{v}_2 - \vetor{v}_1|^p
\label{CoercivityForMathcalAForP2} 
\end{equation}
for any $\vetor{v}_1, \vetor{v}_2 \in \mathbb{R}^n$. If $1 < p \le 2$, then 
\begin{equation} 
|\mathcal{A}(\vetor{v}_2) - \mathcal{A}(\vetor{v}_1)| \le c_2 |\vetor{v}_2 - \vetor{v}_1|^{p-1}
\label{BoundednessForMathcalAForP12} 
\end{equation}
for $\vetor{v}_1, \vetor{v}_2 \in \mathbb{R}^n$.
\label{InequalitiesForCalA}
\end{lemma} 

We can also conclude directly, from \eqref{UpperBoundForA} and \eqref{LowerBoundForA}, that
\begin{equation}
| \mathcal{A}(\vetor{v})| \le D_2 |\vetor{v}|^{p-1} \quad {\rm and} \quad | \mathcal{A}(\vetor{v})| \ge D_1 |\vetor{v}|^{p-1} \quad \text{for } \vetor{v} \in \mathbb{R}^n.
\label{SimpleInequalitiesForMathCalA}
\end{equation}
 We need the following result:
\begin{lemma}
Let $K$ be a compact set of $M$ and $o\in M$ be a fixed point. Suppose that
$U$ and $V$ are bounded domains of $M$ that satisfy $K\subset V\subset\subset
U$. For $R_{2}>R_{1}>0$ such that the ball $B(o,R_{1})$ contains $\overline
{U}$, there exists a function $\eta=\eta_{R_{1},R_{2}}\in W_{0}^{1,p}%
(M\backslash K)$ such that 
\newline$\bullet$ $0\leq\eta\leq1$; 
\newline%
$\bullet$ $\eta(x)=1$ for $x\in B(o,R_{1})\backslash U$; 
\newline$\bullet$
$\eta(x)=0$ if $x\in V$ or $d(x,o)\geq R_{2}$; 
\newline$\bullet$
$|D\eta(x)|\leq m$ for $x\in U\backslash V$, where $m$ is a constant that
depends on $U$, $V$ and $M$, but not on $R_{1}$ and $R_{2}$; 
\newline$\bullet$
$\Delta_{p}\eta(x)=0$ if $R_{1}<d(x,o)<R_{2}$ in the weak sense. 
\label{etaDefinition}
\end{lemma}
\begin{proof} First minimize the functional 
$$J(v):= \int_{B(o,R_2)} |\nabla v|^p dx$$ 
in the convex set
$\{ v \in W_0^{1,p}(B(o,R_2)) \; : \: v =1 \; {\rm in} \; B(o,R_1) \}$.
From the classical theory, there exists a unique minimizer $v_0$ in this convex set
such that $v_0=1$ in $B(o,R_1)$ and $\Delta_p v_0 = 0$ in $B(o,R_2) \backslash B(o,R_1)$ in the weak sense.
Now, for some neighborhood $V_{\varepsilon}$ of $V$ such that $V \subset \subset V_{\varepsilon} \subset \subset U$, let 
$v_1 \in C^{\infty}(U)$ satisfying $0 \le v_1 \le 1$ in $U$, $v_1=0$ in $V$ and $v_1=1$ in $U \backslash V_{\varepsilon}$. Defining $\eta$ by $\eta := v_1$ in $U$, $\eta := v_0$ in $B(o,R_2) \backslash U$ and $\eta:=0$ in $M\backslash B(o,R_2)$, it follows the result.
\end{proof}

Now we fix some notation before stating the main result. For a fixed point $o
\in M$, we denote the capacity of $B(o,R_{1})$ with respect to $B(o,R_{2})$,
for $R_{2} > R_{1}$, by
\[
\mathrm{\; cap}_{p} \, (R_{1},R_{2},o) := \mathrm{\; cap}_{p} \,
(B(o,R_{1});B(o,R_{2})).
\]
If it is clear that the center of the balls is $o$, we simply denote by
$\mathrm{\; cap}_{p} \, (R_{1},R_{2})$.

The annulus $B(o,R_{2}) \backslash\overline{B(o,R_{1})}$, centered at $o$,
is denoted by
\[
\mathpzc{A}_{R_{1},R_{2}}=\mathpzc{A}_{R_{1},R_{2},o} := B(o,R_{2})
\backslash\overline{B(o,R_{1})} \quad\text{for} \quad R_{2} > R_{1}%
\]
and the oscillation of a function $v$ in this annulus is defined by
\begin{equation}
\mathop{\rm osc}_{\mathpzc{A}_{R_{1},R_{2}}} v := \sup_{\mathpzc{A}_{R_{1}%
,R_{2}}} v - \inf_{\mathpzc{A}_{R_{1},R_{2}}} v.\label{oscillationDefinition}%
\end{equation}

The next proposition is a variant of the result that the derivatives of global solutions are in $L^p$ if
the manifold is $p$-parabolic.   

\begin{proposition}
Let $u \in C^{1}(M\backslash K)$ be a weak solution of \eqref{exteriorProblem}
in $M\backslash K$. Suppose that $A$ satisfies \eqref{AdeZeroeZero} -
\eqref{UpperBoundForA} and $U \subset M $ is an open bounded set such that $K
\subset U $. \newline(a) If $u$ is bounded and
\[
\sup_{k} \; \left(\osc{\mathpzc{A}_{R_1^k, R_2^k}} u\right)^p \cdot \capa (R_1^k,R_2^k) < +\infty, \label{GrowthCondOnDirichletIntegral}%
\]
where $B(o,R_{1}^{k})$ and $B(o,R_{2}^{k})$ are increasing sequences of balls
that contain $U$ and such that $R_{2}^{k} > R_{1}^{k} \to\infty$, then $|Du|
\in L^{p}(M \backslash U)$ and there exist positive constants $C_{1}$ and
$C_{2}$, such that
\[
\| Du\|_{L^{p}( M\backslash U)} \le C_{2} + C_{1} \limsup_{k \to\infty} \;
\left( \mathop{\rm osc}_{\mathpzc{A}_{R_{1}^{k}, R_{2}^{k}}} u\right) ^{p}
\cdot\mathrm{\; cap}_{p} \, (R_{1}^{k},R_{2}^{k}).
\]
(b) In the case that $u$ is not necessarily bounded, if we assume
\[
\sup_{k} \quad\left( \max_{\mathpzc{A}_{R_{1}^{k}, R_{2}^{k}}} |u|\right) ^{p}
\cdot\mathrm{\; cap}_{p} \, (R_{1}^{k},R_{2}^{k}) < +\infty,
\]
then $|Du| \in L^{p}(M \backslash U)$ and
\[
\| Du\|_{L^{p}( M\backslash U)} \le C_{2} + C_{1} \limsup_{k \to\infty} \;
\left( \max_{\mathpzc{A}_{R_{1}^{k}, R_{2}^{k}}} |u|\right) ^{p}
\cdot\mathrm{\; cap}_{p} \, (R_{1}^{k},R_{2}^{k}).
\]
(In both cases, $C_{1}$ depends on $p$, $c_{1}$, $D_{2}$ and $C_{2}$ depends
on $p$, $c_{1}$, $D_{2}$, $K$, $U$, and $u$.)
\label{LpBoundnessOfDerivativeOfSolutionMain}
\end{proposition}

\begin{proof}
{\bf (a)} Let $V$ be an open set and $R_0 >0$ s.t.
$K \subset V \subset \subset U \subset \subset B(o,R_0).$
For $R_1 > R_0$ and $R_2 > R_0$, consider the function $\eta=\eta_{R_1,R_2}$ as in Lemma \ref{etaDefinition} associated to $V$, $U$, $R_1$ and $R_2$.
Note that the function
$$\varphi := \eta^p (u - I), \quad {\rm where} \quad I = \inf_{\mathpzc{A}_{R_1, R_2}} u,$$
belongs to $W_0^{1,p} (M \backslash K)$.
Hence, using that
$u$ is a weak solution of \eqref{exteriorProblem},
we obtain
$$\int_F (u-I) \nabla \eta^p \cdot \mathcal{A}(\nabla u ) \, dx +
\int_F \eta^p \, \nabla u \cdot \mathcal{A}(\nabla u) \, dx = 0, $$
where $F \subset B(o,R_2) \backslash V$
is the compact support of $\eta$.
Hence,
\eqref{GenericInequalityFromBelowForA} and \eqref{SimpleInequalitiesForMathCalA}
imply that
\begin{align*}
c_1 \int_F \eta^p |\nabla u|^p dx &\le
\int_F \eta^p \nabla u \cdot \mathcal{A}(\nabla u) dx \\[5pt]
& = - \int_F (u-I) \nabla \eta^p \cdot \mathcal{A}(\nabla u ) dx\\[5pt]
& \le \int_F p |u-I| \eta^{p-1}|\nabla \eta| | \mathcal{A}(\nabla u ) | dx \\[5pt]
& \le p \, D_2  \int_F |u-I| \eta^{p-1}|\nabla \eta| |\nabla u|^{p-1} dx .
\end{align*}
From H\"older inequality,
\begin{equation}
c_1 \int_F \eta^p |\nabla u|^p dx  \le  p \, D_2
\left(\int_F \eta^{p} |\nabla u|^{p} dx\right)^{\frac{p-1}{p}}
\left( \int_F |u-I|^p |\nabla \eta|^p dx \right)^{\frac{1}{p}},
\label{0thLpBoundOnU}
\end{equation}
that is,
\begin{equation}
\int_F \eta^p |\nabla u|^p dx  \le
\left(\frac{p \, D_2 }{c_1} \right)^p  \int_F |u-I|^p |\nabla \eta|^p dx .
\label{firstLpBoundOnU}
\end{equation}
Using the hypotheses on the derivative of $\eta$,
we obtain
\begin{align}
\int_F |u-I|^p |\nabla \eta|^p dx &=
\int_{U\backslash V} |u-I|^p |\nabla \eta|^p dx +
\int_{\mathpzc{A}_{R_1,R_2}}  |u-I|^p |\nabla \eta|^p dx \nonumber \\[5pt]
& \le 2^pS_0^p \; m^p \, \text{Vol} (U\backslash V ) \,+\, \left(\osc{\mathpzc{A}_{R_1, R_2}} u \right)^p \; \int_{\mathpzc{A}_{R_1,R_2}} |\nabla \eta|^p dx,
\label{secondLpBoundOnU}
\end{align}
where $S_0=\sup_{M \backslash V} |u| $.
Since $\eta$ is $p$-harmonic in $\mathpzc{A}_{R_1,R_2}$, $\eta=1$ on $\partial B(o,R_1)$ and $\eta = 0$ on $\partial B(o,R_2)$, then
$$ \int_{\mathpzc{A}_{R_1,R_2}} |\nabla \eta|^p dx  = \inf_{v \in \mathcal{F}_{B(o,R_1),B(o,R_2)}} \int_{M} |\nabla v|^p dx = \capa(R_1,R_2).$$
Hence, from \eqref{firstLpBoundOnU} and \eqref{secondLpBoundOnU}, we conclude
that
\begin{equation*} \int_F \eta^p \, |\nabla u|^p \,dx  \le
C_1 \; 2^p \; S_0^p \; m^p \, \text{Vol} (U\backslash V ) + C_1 \; \left(\osc{\mathpzc{A}_{R_1, R_2}} u \right)^p \; \capa(R_1,R_2),
\end{equation*}
where $C_1 = \left(p \, D_2 / c_1 \right)^p$. Since
$B(o,R_1) \backslash U \subset F$ and
$\eta=1$ in $B(o,R_1) \backslash U$,
$$ \int_{B(o,R_1) \backslash U}  |\nabla u|^p \, dx
\,\le\, C_2 + C_1 \; \left(\osc{\mathpzc{A}_{R_1, R_2}} u \right)^p \; \capa(R_1,R_2), $$
where $C_2= C_1 \; 2^p \; S_0^p \; m^p \, \text{Vol} (U\backslash V )$.
In particular, this holds for $R_1^k$ and $R_2^k$.
Then, according to the hypotheses, the right-hand side is bounded, proving (a).

\vspace{.2cm}
\noindent {\bf (b)} Taking the test function $\varphi := \eta^p u$ and applying the same argument as in (a), we conclude (b). The main difference in relation to (a) is that we have to replace $S_0$ by $\sup_{U \backslash V} |u|$ and $\osc{\mathpzc{A}_{R_1, R_2}} u$ by $\max_{\mathpzc{A}_{R_1^k, R_2^k}} |u|$.
\end{proof}

The idea of the proof of the next theorem is based on the
proof of Theorem 2 of \cite{BSZ}.

\begin{theorem}
Let $u,v \in C(\overline{M \backslash K}) \cap C^{1}(M \backslash K)$ be weak
solutions of \eqref{exteriorProblem} in $M\backslash K$, where $A$ satisfies
\eqref{AdeZeroeZero} - \eqref{UpperBoundForA} and $p >1$. Assume also that
\begin{equation}
\sup_{k} \left( \max_{\mathpzc{A}_{R_{1}^{k}, R_{2}^{k}}} |u|\right) ^{p} \!\!
\mathrm{\; cap}_{p} \, (R_{1}^{k},R_{2}^{k}) \quad\mathrm{and} \quad\sup_{k}
\left( \max_{\mathpzc{A}_{R_{1}^{k}, R_{2}^{k}}} |v|\right) ^{p} \!\!
\mathrm{\; cap}_{p} \, (R_{1}^{k},R_{2}^{k})\label{boundedCapacityHypotheses}%
\end{equation}
are finite for some sequences $(R_{1}^{k})$ and $(R_{2}^{k})$ such that
$R_{2}^{k} > R_{1}^{k} \to\infty$. 
\newline If $u \le v$ on $\partial K$, then
\[
u \le v \quad\text{in} \quad M \backslash K.
\]
\label{GenericComparisonTheorem}
\end{theorem}

\begin{proof}
Let $(\varepsilon_k)$ be a sequence of positive numbers such that $\varepsilon_k \downarrow 0$ and
\begin{equation}
\lim_{k \to \infty} \varepsilon_k^p \; \capa (R_1^k,R_2^k) = 0.
\label{varepsilonKgoesToZero}
\end{equation}
Observe that $v-u + \varepsilon_k \ge \varepsilon_k $ on $\partial K$ for any $k$.
Then, using that $v-u + \varepsilon_k$ is continuous, we conclude that
there exists a bounded open set $U =U_k \supset K$
such that $ v-u + \varepsilon_k > 0$ in $U \backslash K$. We can suppose that $U_k$ is a decreasing sequence and $\bigcap_{k=1}^{\infty} U_k = K$.
As in the previous proposition, let $V=V_k$ be an open set and $R_0 \geq 1$ such that $ K \subset V_k \subset \subset U_k \subset U_1 \subset B(o,R_0)$.
For $R_2^k > R_1^k > R_0$, let $\eta=\eta_{R_1^k, R_2^k}$
as described in Lemma \ref{etaDefinition}.
Observe that
$$(v-u + \varepsilon_k)^- \;\!=\, \max \,\{-(v-u+\varepsilon_k),0\}
\:\!=\, 0
\quad {\rm in} \quad U \backslash K.$$
Then
$ \varphi : = \eta^p \;\!(v-u+\varepsilon_k)^-$ has a compact support
$$ F_k \subset B(o,R_2^k) \backslash U_k \subset M\backslash K.$$
Hence, using that
$ v-u + \varepsilon_k \in C^1(M\backslash K)$,
we have that $\varphi \in  W^{1,p}_0(M \backslash K)$.
Since $u$ and $v$ are weak solutions
of \eqref{exteriorProblem} and
$\varphi \in W^{1,p}_0(M \backslash K)$,
we get
\begin{equation} \int_{\mbox{}_{\scriptstyle M \backslash \;\!K}}
\hspace{-0.550cm}
\nabla \varphi \;\!\cdot [\;\! {\cal A}( \nabla v) -  {\cal A}(\nabla u)\;\!]
\, d M \,=\, 0.
\label{weakSolutionForUandV}
\end{equation}
Observe that $F_k \subset \{ x \in M \backslash K \, : \, v(x) - u(x) + \varepsilon_k \le 0\}$ and, therefore,
$$\nabla \varphi =\;\! \eta^p \;\!
\chi_{\mbox{}_{\scriptstyle F_k}} \!\!\nabla (u-v) \;\!+
p \;\!\;\! \eta^{p-1} (v-u+\varepsilon_k)^- \;\!\nabla \eta \quad \text{a.e. in } M \backslash K, $$
where
$ \:\!\chi_{\mbox{}_{\scriptstyle F_k}} \!\!\:\!$
is the characteristic function of
the set $F_k$.
Hence, it follows that
\begin{align}
\int_{\mbox{}_{\scriptstyle F_k}}
\hspace{-0.100cm}
\eta^p \, & \nabla (v-u) \cdot
[\;\! {\cal A}(\nabla v) -  {\cal A}(\nabla u)\;\!] \,dM \;\!\;\!= \nonumber \\[5pt]
&
\int_{\mbox{}_{\scriptstyle \mathpzc{A}_{R_1^k,R_2^k}}}
\hspace{-0.150cm}
p \;\!\;\! \eta^{p-1} \;\!(v-u+\varepsilon_k)^- \,
\nabla \eta
\cdot [\;\! {\cal A}(\nabla v) -  {\cal A}(\nabla u) \;\!] \,dM.
\label{weakSolAppliedInATestFunc}
\end{align}
Replacing $F_k$ by $G_k:=F_k \cap \{ x \in M \backslash K \; : \; |\nabla u(x)| + |\nabla v(x)| \ne 0\}$, the left-hand side of \eqref{weakSolAppliedInATestFunc} does not change.
Then, from \eqref{GenericInequalityFromBelowForA} and \eqref{SimpleInequalitiesForMathCalA}, we have
\begin{align}
c_1 \!
\int_{\mbox{}_{\scriptstyle G_k}}
\hspace{-0.100cm}
\eta^p \, &(|\nabla v| + |\nabla u|)^{p-2} | \nabla (v-u) |^2  \,dM \,
\le \nonumber \\[5pt]
&\le \int_{\mbox{}_{\scriptstyle \mathpzc{A}_{R_1^k,R_2^k}}} \hspace{-0.150cm}
p \;\!\;\! \eta^{p-1} \,(v-u+\varepsilon_k)^- \,
\nabla \eta \cdot
[\;\! \mathcal{A}(\nabla v) -  \mathcal{A}(\nabla u) \;\!] \,dM. \\[5pt]
& \le \,
D_2 \;\! p \!\:\!
\int_{\mbox{}_{\scriptstyle \mathpzc{A}_{R_1^k,R_2^k}}} \hspace{-0.150cm}
\eta^{p-1} \;\! S_k \, |\nabla \eta| \,
(\;\!|\nabla v |^{\:\!p-1} \!\;\!+\;\! |\nabla u|^{\:\!p-1})
\, dM,
\label{inequalityForTheDifferenceOfSolutions1}
\end{align}
where
$$ S_k = \max_{\mathpzc{A}_{R_1^k,R_2^k}} \,(v-u+\varepsilon_k)^- \le \;\!\max_{\mathpzc{A}_{R_1^k,R_2^k}}|u| + \max_{\mathpzc{A}_{R_1^k,R_2^k}} |v| + \varepsilon_k < \infty .$$
Using  H\"older's inequality,
we have
\begin{align*}
c_1 \!
\int_{\mbox{}_{\scriptstyle G_k}}
\hspace{-0.100cm}
& \eta^p \,  (|\nabla v| + |\nabla u|)^{p-2} | \nabla (v-u) |^2  \,dM \\[5pt]
& \le \,
D_2 \;\!p \!\!\;\!
\int_{\mbox{}_{\scriptstyle \mathpzc{A}_{R_1^k,R_2^k}}} \hspace{-0.250cm}
S_k \, |\nabla \eta| \;\!\;\! \eta^{\:\!p-1} \,\!
2 \max \,\{\:\!|\nabla u |, \;\!|\nabla v|\:\!\}^{p-1} \,dM \\[5pt]
& \le\,
2 D_2 \;\!p\,  \left(
\int_{\mbox{}_{\scriptstyle \mathpzc{A}_{R_1^k,R_2^k}}} \hspace{-0.250cm}
S_k^p |\nabla \eta|^p \;\!dM\;\!\;\!
\right)^{\!\!\!\;\!\frac{1}{p}}
\!\:\!
\left(
\int_{\mbox{}_{\scriptstyle \mathpzc{A}_{R_1^k,R_2^k}}} \hspace{-0.250cm}
\eta^{\:\!p} \;\!(\:\!|\nabla u |^{\:\!p} \!\;\!+\:\! |\nabla v|^{\:\!p}\:\!)
\,dM
\right)^{\!\!\!\!\;\!\frac{p-1}{p}} \!\!.
\end{align*}
Therefore,
since $ |\eta| \le 1 $ and $ p > 1 $,
\begin{align}
\int_{\mbox{}_{\scriptstyle G_k}}
\hspace{-0.100cm}
\eta^p \,  (|\nabla v| + |\nabla u|)^{p-2} & | \nabla (v-u) |^2  \,dM
\le\;
\frac{\: 2 D_2 \:\!p\;}{c_1} \;\!
\left(
\int_{\mbox{}_{\scriptstyle \mathpzc{A}_{R_1^k,R_2^k}}} \hspace{-0.250cm}
S_k^p \, |\nabla \eta|^{\:\!p} \, dM
\,
\right)^{\!\!\!\;\!\frac{1}{p}} \nonumber \\[5pt]
& \times \left(\;\!
\|\;\!\nabla u\;\!
\|_{\mbox{}_{\scriptstyle L^{p}(\mathpzc{A}_{R_1^k,R_2^k})}}
^{\mbox{}^{\scriptstyle p - 1}}
\!+\,
\|\;\!\nabla v \;\!
\|_{\mbox{}_{\scriptstyle L^{p}(\mathpzc{A}_{R_1^k,R_2^k})}}
^{\mbox{}^{\scriptstyle p - 1}}
\:\!\right)
\!\:\!.
\label{comparisonTheoremIneq1}
\end{align}
Observe now that
\begin{align}
\int_{\mbox{}_{\scriptstyle \mathpzc{A}_{R_1^k,R_2^k}}} \hspace{-0.250cm}
S_k^p \, |\nabla \eta|^{\:\!p} \, dM
&\le\;
\left[\max_{\mathpzc{A}_{R_1^k,R_2^k}} \,(v-u+\varepsilon_k)^- \right]^p  \int_{\mbox{}_{\scriptstyle \mathpzc{A}_{R_1^k,R_2^k}}} \hspace{-0.250cm}
|\nabla \eta|^{\:\!p} \, dM \nonumber \\[5pt]
&\le \; 3^p \left[ \max_{\mathpzc{A}_{R_1^k,R_2^k}} |v|^p + \max_{\mathpzc{A}_{R_1^k,R_2^k}} |u|^p + \varepsilon_k^p \right]  \capa (R_1^k,R_2^k)
\label{comparisonTheoremIneq2}
\end{align}
for all $k$.
From this and \eqref{comparisonTheoremIneq1}, we conclude that%
\begin{align}
\int_{\mbox{}_{\scriptstyle G_k}}
\hspace{-0.100cm}
& \eta^p \,  (|\nabla v| + |\nabla u|)^{p-2}  | \nabla (v-u) |^2  \,dM   \le
\frac{\:2 D_2 \:\!p\; 3^p}{c_1}
\;\!\;\!
\left[ \max_{\mathpzc{A}_{R_1^k,R_2^k}} |v|^p +  \right. \nonumber \\[5pt]
&\; \left. \max_{\mathpzc{A}_{R_1^k,R_2^k}} |u|^p + \varepsilon_k^p \right]  \capa (R_1^k,R_2^k)
\left(\;\!
\|\;\!\nabla u\;\!
\|_{\mbox{}_{\scriptstyle L^{p}(\mathpzc{A}_{R_1^k,R_2^k})}}
^{\mbox{}^{\scriptstyle p - 1}}
\!+\,
\|\;\!\nabla v \;\!
\|_{\mbox{}_{\scriptstyle L^{p}(\mathpzc{A}_{R_1^k,R_2^k})}}
^{\mbox{}^{\scriptstyle p - 1}}
\:\!\right).
\label{EstimateOfDiffOfSolByPhiAndGrad}
\end{align}
Note that the last two terms in this inequality converges to zero, according to Proposition \ref{LpBoundnessOfDerivativeOfSolutionMain}. Moreover,
$$\max_{\mathpzc{A}_{R_1^k,R_2^k}} |v|^p \capa (R_1^k,R_2^k) \quad {\rm and} \quad \max_{\mathpzc{A}_{R_1^k,R_2^k}} |u|^p \capa (R_1^k,R_2^k) $$ are bounded from hypothesis and $\varepsilon_k^p \capa (R_1^k,R_2^k)$ is bounded from \eqref{varepsilonKgoesToZero}. Hence, the right-hand side of \eqref{EstimateOfDiffOfSolByPhiAndGrad}
converges to $0$ as $k \to \infty$.
On the other hand,
$$
\int_{\mbox{}_{\scriptstyle G_k}}
\hspace{-0.300cm}
\eta^p_{R_1^k,R_2^k} \, (|\nabla v| + |\nabla u|)^{p-2}  | \nabla (v-u) |^2 \, dM
\;\to
\int_{\mbox{}_{\scriptstyle G_0}}
\hspace{-0.300cm}
(|\nabla v| + |\nabla u|)^{p-2}  | \nabla (v-u) |^2 \,dM
$$
as $k \to \infty$, where
$$G_0=\{ x \in M \backslash K \, : \, v(x) < u(x) \} \cap \{ x \in M \; \backslash K : \; |\nabla u(x)| + |\nabla v(x)| \ne 0\} $$ since $\eta_{R_1^k,R_2^k} \to 1$ in $M \backslash K$ as $k \to \infty$ and $G_k$ is an increasing sequence ($F_k$ is an increasing sequence) of sets such that $\bigcup G_k = G_0$. Then
\begin{equation}
\int_{\mbox{}_{\scriptstyle G_0}}
\hspace{-0.300cm}
(|\nabla v| + |\nabla u|)^{p-2}  | \nabla (v-u) |^2 \,dM \;\!\;\!=\;\!\;\! 0.
\label{NullDirichletIntegralInF}
\end{equation}
Therefore, $\chi_{\mbox{}_{\scriptstyle G_0}} \nabla (v-u) =0$. Since
$$ \nabla ( v - u )^{-}
\!\;\!=\;\!
\chi_{\mbox{}_{\scriptstyle \{ x \in M \backslash K \; : \; v(x) < u(x) \} }}
\!\!
\nabla (u - v) \!\;\!=\;\!
\chi_{\mbox{}_{\scriptstyle G_0}}
\!\!
\nabla (u - v) \quad {\rm a.e.\; \;in} \; \;  M \backslash K ,
$$
it follows that $\nabla ( v - u )^{-} = 0$ a.e. in $M \backslash K$.
Using this and that
$ (v - u)^{-} \!\:\!=\:\!0 \:\!$
on $\partial K $,
we conclude that
$ (v - u)^{-} \!\:\!=\:\! 0 $
in $M \backslash K $.
Therefore,
$$ u(x) \,\leq\, v(x)
\quad \text{for} \quad x \in M \backslash K, $$
proving the result.
\end{proof}

In particular, if $u$ and $v$ are bounded, the control of $\mathrm{\; cap}_{p}
\, (R_{1}^{k},R_{2}^{k})$ guarantees the comparison principle for exterior domains:

\begin{corollary}
Let $M$ be a complete noncompact Riemannian manifold. Let $A$ be a function that
satisfies \eqref{AdeZeroeZero} - \eqref{UpperBoundForA} and $p > 1$. If
\[
\sup_{k} \mathrm{\; cap}_{p} \, (R_{1}^{k},R_{2}^{k}) < +\infty,
\]
then the comparison principle holds for the exterior problem
\eqref{exteriorProblem}. \label{genericComparisonResultForBoundedSol}
\end{corollary}

\ 

\begin{remark}
\label{comparisonRemarkForPEqualTo2}
For $p=2$, Theorem \ref{GenericComparisonTheorem} holds even if the terms in
\eqref{boundedCapacityHypotheses} are not necessarily bounded, provided that
\begin{equation}
\max_{\mathpzc{A}_{R_{1}^{k}, R_{2}^{k}}} |v - u|^{p} \mathrm{\; cap}_{p} \,
(R_{1}^{k},R_{2}^{k}) \to0 \quad\mathrm{as} \quad k \to+\infty
.\label{differenceBetweenUandVGoesToZero}%
\end{equation}

\ 

\noindent Indeed, according to \eqref{weakSolAppliedInATestFunc},
\eqref{CoercivityForMathcalAForP2}, \eqref{BoundednessForMathcalAForP12} and
the definition of $F_{k}$, for $p=2$,
\begin{align*}
c_{1} \! \int_{\mbox{}_{\scriptstyle F_{k}}} \hspace{-0.100cm} \eta^{2} \, |
\nabla(v-u) |^{2} \;\!dM \,  & \le\, 2 \; c_{2} \!\:\! \int
_{\mbox{}_{\scriptstyle \mathpzc{A}_{R_{1}^{k},R_{2}^{k}} \cap F_{k}} }
\hspace{-0.150cm} \eta\;\! S_{k} \, |\nabla\eta| \, |\nabla v - \nabla u| \,
dM\\[5pt]
& \le\left(  \int_{\mbox{}_{\scriptstyle F_{k}}} \hspace{-0.100cm} \eta^{2} \,
| \nabla(v-u) |^{2} \;\!dM \right) ^{\frac{1}{2}} \left(  \int
_{\mbox{}_{\scriptstyle \mathpzc{A}_{R_{1}^{k},R_{2}^{k}}}} \hspace{-0.250cm}
S_{k}^{2} |\nabla\eta|^{2} \;\!dM \right) ^{\frac{1}{2}}.
\end{align*}
Therefore
\begin{align*}
\int_{\mbox{}_{\scriptstyle F_{k}}} \hspace{-0.100cm} \eta^{2} \, |
\nabla(v-u) |^{2} \;\!dM \,  & \le\, \left(  \frac{2 \; c_{2}}{c_{1}} \right)
^{2} \int_{\mbox{}_{\scriptstyle \mathpzc{A}_{R_{1}^{k},R_{2}^{k}}}}
\hspace{-0.250cm} S_{k}^{2} |\nabla\eta|^{2} \;\!dM .
\end{align*}
Doing the same computation as in \eqref{comparisonTheoremIneq2}, we have
\[
\int_{\mbox{}_{\scriptstyle \mathpzc{A}_{R_{1}^{k},R_{2}^{k}}}} \hspace
{-0.250cm} S_{k}^{2} \, |\nabla\eta|^{\:\!2} \, dM \le\; 2^{2} \left[
\sup_{\mathpzc{A}_{R_{1}^{k},R_{2}^{k}}} |v - u|^{p} + \varepsilon_{k}^{p}
\right]  \mathrm{\; cap}_{p} \, (R_{1}^{k},R_{2}^{k})
\]
and, therefore,
\[
\int_{\mbox{}_{\scriptstyle F_{k}}} \hspace{-0.100cm} \eta^{2} \, |
\nabla(v-u) |^{2} \;\!dM \, \le\, \left(  \frac{2 \; c_{2}}{c_{1}} \right)
^{2} 4 \left[  \sup_{\mathpzc{A}_{R_{1}^{k},R_{2}^{k}}} |v - u|^{p} +
\varepsilon_{k}^{p} \right]  \!\! \mathrm{\; cap}_{p} \, (R_{1}^{k},R_{2}%
^{k}).
\]
Hence, using \eqref{differenceBetweenUandVGoesToZero} and
\eqref{varepsilonKgoesToZero}, we have that
\[
\int_{\mbox{}_{\scriptstyle F_{k}}} \hspace{-0.100cm} \eta_{R_{1}^{k}%
,R_{2}^{k}}^{2} \, | \nabla(v-u) |^{2} \;\!dM \to0.
\]
Then, as in the theorem, relation \eqref{NullDirichletIntegralInF} holds and
following the same argument as before, we conclude that $u \le v$.
\end{remark}

\section{Equivalence between the $p-$parabolicity and the $p-$comparison principle}

In this section we prove that the $p$-comparison principle, as defined previously,
holds for the exterior problem \eqref{exteriorProblem} if and only if $M$ is
$p$-parabolic. 

First note that
\begin{equation}
\mathrm{\; Cap}_{p} \,(E) \le\mathrm{\; cap}_{p} \,(E;\Omega) \quad
\mathrm{and} \quad\mathrm{\; Cap}_{p} \,(E) \le\mathrm{\; Cap}_{p} \,(F)
\label{relForCapacities}%
 \end{equation}
for any $\Omega\subset M$ and $E \subset F \subset M$, since $\mathcal{F}_{E,\Omega} \subset \mathcal{F}_{E,M}$ and
$\mathcal{F}_{E,M} \supset \mathcal{F}_{F,M}$ in this case.

We also need the next result, that is an extension of the one established in Corollary 4.6 of \cite{Gr0} to the case $p=2$.

\begin{lemma}
If $U \subset M$ is a bounded domain with a $C^{2}$ boundary and $\mathrm{\;
Cap}_{p} \, (U) > 0$, then there exists some non-constant $p$-harmonic
function $u$ such that $u=1$ on $\partial U$ and $0 < u < 1$ in $M \backslash
U$. \label{existenceOfNonConstantp-harmonic}
\end{lemma}

\begin{proof}
Let $(W_k)$ be an increasing sequence of bounded domains with $C^2$ boundary such that
$$  \bigcup_{k=1}^{\infty} W_k = M \quad {\rm and } \quad  U \subset \subset W_k \subset \subset W_{k+1}  \quad {\rm for \; any} \; k \in \mathbb{N}. $$
(If there exists an increasing sequence of balls $B(o,R_k)$ with $C^2$ boundary, where $R_k \to +\infty$, we can take $W_k=B(o,R_k)$.) From the theory for PDE, there exists a function $u_k \in \mathcal{F}_{U,W_k}$ such that
\begin{equation}
\int_{M} |\nabla u_k|^p \, dx = \inf_{v \in \mathcal{F}_{U,W_k}} \int_{M} |\nabla v|^p dx = \capa (U; W_k).
\label{relCapacityOfUandWk}
\end{equation}
Moreover, $u_k$ is the $p$-harmonic function in the $A_{k}=A_{U,W_k}:=W_k \backslash \overline{U}$ such that $$u_k=1 \quad {\rm on} \quad \partial U \quad {\rm and} \quad  u_k=0 \quad {\rm on} \quad \partial W_k \quad {\rm in \; \; the \; \; trace \; \; sense.}$$
Due to smoothness of $\partial U$ and $\partial W_k$, the theory of regularity implies that $u_k \in C^1(\overline{A}_{k})$ (for instance, see \cite{HMP}).
From the strong maximum principle $0 < u_k < 1$ in $A_{k}$. Hence $u_{k} >0$ on $\partial W_{k-1}$ for any $k \ge 2$, since $\partial W_{k-1} \subset A_{k} = W_k \backslash \overline{U}$. Therefore, using the comparison principle, we conclude that
\begin{equation}
u_{k-1}  < u_{k} \quad {\rm in} \quad A_{k-1}.
\label{monotonicityOfu-R}
\end{equation}
Observe also that $u_k =1$ in $\overline{U}$ and $u_k=0$ in $M \backslash W_k$, since $u_k \in \mathcal{F}_{U,W_k}$.
(This implies that $u_k$ is continuous in $M$ due to the fact that $u_k \in C^1(\overline{A}_k)$.)
Then, from \eqref{monotonicityOfu-R}, we have that
$$  u_{k-1}  \le u_{k} \quad {\rm in} \quad M  \quad {\rm for } \quad k \in \{2,3,\dots \}.$$
Hence, using that $0 \le u_{k} \le 1$ for any $k$, we have that $u_k$ converges to some function $u$ defined in $M$ satisfying $0 \le u \le 1$. In particular, $u=1$ in $\overline{U}$.
Note also that $u \ge u_{k} > 0$ in $W_k$ for any $k$. Thus, $u > 0$ in $M$.
\\ \\ $\bullet$ {\it Statement 1}: $u$ is $p$-harmonic.
Observe that for any bounded domain $V \subset \subset M \backslash \overline{U}$, we have that $\overline{V} \subset A_k$ for large $k$.
Then, starting from some large $k$, $(u_k)$ is a uniformly bounded sequence of $p$-harmonic functions in $V$.
Hence, according to Theorem 1.1 of \cite{WZ}, there exists some constant $C >0$ that depends on $V$, $M$, $n$ and $p$ such that $|\nabla u_k| \le C$ in $\overline{V}$. Thus, $u_k$ is uniformly bounded in $W^{1,p}(V)$ and, therefore, up to a subsequence, we have that
$$u_k \rightharpoonup v_0 \quad {\rm in } \quad W^{1,p}(V) \quad {\rm and} \quad u_k \to v_0 \quad {\rm in} \quad L^p(V)$$
for some $v_0 \in W^{1,p}(V)$, due to the reflexivity of $W^{1,p}(V)$ and the Rellich-Kondrachov theorem.
Indeed, $v_0=u$ since $u_k \to u$ pointwise.
Hence, using that $u_k$ is $p$-harmonic and the same argument as in \cite{Ev} (see Theorem 3 of page 495), we conclude that
$$0=\lim_{k \to \infty} \int_{V} |\nabla u_k|^{p-2} \nabla u_k \nabla \varphi \, dx = \int_{V} |\nabla u|^{p-2} \nabla u \nabla \varphi \, dx, $$
for any $\varphi \in C_c^{\infty} (V)$, that is, $u$ is $p$-harmonic in $V$, proving the statement.
\\ \\ $\bullet$ {\it Statement 2}: $u$ is non-constant and $0< u <1$ in $M \backslash \overline{U}$.
Since we already proved that $0 < u \le 1$ in $M\backslash \overline{U}$, from the maximum principle, we just need to show that $u$ is non-constant.
Suppose that $u$ is constant. Let $B^{*}=B(o,R_0)$ be an open ball such that $\overline{U} \subset \subset B^{*}$ and $\rho \in C^{\infty}(M)$ be a function that satisfies $0 \le \rho \le 1$, $\rho =0$ in some neighborhood of $\overline{U}$ and $\rho =1$ in $M \backslash \overline{B^{*}}$. (We can suppose w.l.g. that $\overline{B^{*}} \subset W_k$ for any $k$.) Since $u_k$ is $p$-harmonic in $A_k$ and $\varphi = \rho^p u_k \in W^{1,p}_0(A_k)$, we have
$$-\int_{A_k} p \, u_k \, \rho^{p-1} |\nabla u_k|^{p-2} \nabla u_k \nabla \rho \, dx =  \int_{A_k}  \rho^{p} |\nabla u_k|^{p}  \, dx \ge \int_{M \backslash \overline{B^*}} |\nabla u_k|^{p}  \, dx. $$
Hence, using \eqref{relCapacityOfUandWk} and \eqref{relForCapacities}, we conclude that
\begin{align*}
-\int_{A_k} p \, u_k \, \rho^{p-1} |\nabla u_k|^{p-2} \nabla u_k \nabla \rho \, dx &+ \int_{\overline{B^*}} |\nabla u_k|^{p}  \, dx \\[5pt]
&\ge \int_{M \backslash \overline{B^*}} |\nabla u_k|^{p}  \, dx + \int_{\overline{B^*}} |\nabla u_k|^{p}  \, dx \\[5pt]
&= \capa (U; W_k) \ge \Capa (U) > 0.
\end{align*}
Then, from the fact that $\nabla \rho =0$ in $M\backslash \overline{B^*}$ and $\nabla u_k =0$ in $U$, we have
\begin{equation}
-\int_{\overline{B^*} \backslash U} p \, u_k \, \rho^{p-1} |\nabla u_k|^{p-2} \nabla u_k \nabla \rho \, dx +
\int_{\overline{B^*} \backslash U} |\nabla u_k|^{p}  \, dx \ge \Capa (U).
\label{sumOfIntegralsBoundedFromBelow}
\end{equation}
Now the idea is to show that the left-hand side goes to zero as $k \to +\infty$, leading a contradiction. For that, observe that since $u_1$ is $C^1(\overline{A}_1)$, $u_1=1$ on $\partial U$, $\partial U$ is $C^2$ and $0< u_1 < 1$ in $A_1$, there exists $c_1 > 0$ such that
$$ 1 - c_1 \, dist(x, \partial U) \le u_1(x) < 1 \quad {\rm for } \quad  x \in A_1.$$
Using that $0 < u_1 \le u_k < 1$ in $A_1$ for any $k$, we have
\begin{equation}
0 < 1 - u_k(x) \le \min \{ 1, c_1 \, dist(x, \partial U) \}  \quad {\rm for } \quad  x \in A_1.
\label{boundControlOfuk}
\end{equation}
Furthermore, the inclusion $\overline{B^*} \subset W_1$ implies that there exists $r_0 > 0$ such that $B(x,r_0) \subset W_1$ for any $x \in \overline{B^*}$. Therefore, for any $x \in \overline{B^*}\backslash \overline{U}$ it follows that
$$ B(x,r) \subset W_1 \backslash \overline{U}=A_1 \quad {\rm if} \quad r \le r_1:= \min \{ r_0, dist(x,\partial U)\}.$$
Hence, using that $1-u_k$ is positive and $p$-harmonic in $A_1$, Theorem 1.1 of \cite{WZ} and \eqref{boundControlOfuk}, we have
$$ |\nabla u_k(x)| \le \tilde{C} \frac{(1-u_k(x))}{r_1} \le \tilde{C} \frac{\min \{ 1, c_1 \, dist(x, \partial U) \}}{r_1} \le \tilde{C} \max \left\{ c_1 , 1/r_0 \right\}, $$
for $x \in \overline{B^*}\backslash \overline{U}$, where $\tilde{C} >0$ is a constant that depends on $n$, $p$, $M$ and $A_1$.
Therefore, the sequence $|\nabla u_k|$ is uniformly bounded in $\overline{B^*}\backslash \overline{U}$.
Hence, if we prove that $|\nabla u_k|$ converges to zero pointwise in $\overline{B^*}\backslash \overline{U}$, then the bounded convergence theorem implies that the left-hand side of \eqref{sumOfIntegralsBoundedFromBelow} converges to zero generating a contradiction.
For that, observe that $u = 1$ in $M$, since we are assuming that $u$ is a constant and $u=1$ in $U$. From the fact that $u_k \to u$, we conclude that $1-u_k \to 0$. Then, using that $ B(x,r_1) \subset A_1$ for any $x \in \overline{B^*}\backslash \overline{U}$ and
$$|\nabla u_k(x)| \le \tilde{C} \frac{(1-u_k(x))}{r_1}$$
as before, it follows that $|\nabla u_k(x)| \to 0$ as $k \to +\infty$ for $x \in \overline{B^*}\backslash \overline{U}$.
Therefore, from the bounded convergence theorem we have that left-hand side of \eqref{sumOfIntegralsBoundedFromBelow} converges to zero, contradicting $\Capa (U) > 0$.
\end{proof}

The following result  is a consequence of Corollary \ref{Dichotomy}.

\begin{lemma}
Let $M$ be a complete noncompact Riemannian manifold, $p > 1$ and
$B_{0}=B(o,R_{0})$ some open ball in $M$. Then $\mathrm{\; Cap}_{p} \, (B_{0})
=0$ if and only if there exist two sequences $R_{1}^{k}$ and $R_{2}^{k}$
such that $R_{2}^{k} > R_{1}^{k} \to+\infty$ and $\mathrm{\; cap}_{p} \,
(R_{1}^{k},R_{2}^{k},o) \to0$. \label{SecondCharacterizationForP-parabolic}
\end{lemma}

\begin{proof}
Suppose that $\Capa (B_0) =0$. Let $(R_1^k)$ be an increasing sequence such that $R_1^k \to +\infty$. Then, from Corollary \ref{Dichotomy}, $\Capa (B(o,R_1^k))=0$ for any $k$. Moreover, for a fixed $k$ and any sequence $(R_j)$ satisfying $R_j \to +\infty$, \eqref{CapaIsLimitOfcapa} implies that
$$ 0= \Capa (B(o,R_1^k)) = \lim_{j \to +\infty} \capa (B(o,R_1^k), B(o,R_j))$$
Hence, there exists some $R_2^k > R_1^k$ such that
$$ \capa (B(o,R_1^k), B(o,R_2^k)) < 1/k,$$
proving that $\capa (R_1^k,R_2^k,o) \to 0$.
Reciprocally, assume that $\capa (R_1^k,R_2^k,o) \to 0$, where $R_2^k > R_1^k \to +\infty$.
From \eqref{relForCapacities}, for $R_1^k > R_1^1$, we have
$$\Capa (B(o,R_1^1)) \le \Capa (B(o,R_1^k)) \le \capa (R_1^k,R_2^k,o) \to 0.$$
Therefore, Corollary \ref{Dichotomy} implies that $\Capa (B_0)=0$.
\end{proof}

Combining this lemma with the results of the previous section, we obtain a comparison principle for
exterior problems where the operator involved is more general than the $p$-laplacian operator. Since the comparison principle is an important issue in the PDE theory, the following result that holds for a larger class of operators
might  be interesting by itself. 

\begin{theorem}
Let $M$ be a complete noncompact Riemannian manifold, $A$ be a function that
satisfies \eqref{AdeZeroeZero} - \eqref{UpperBoundForA} and $p > 1$. If $M$ is
$p$-parabolic, then the comparison principle holds for the exterior problem
\eqref{exteriorProblem}. \label{generalResultForP-parabolicManifold}
\end{theorem}

\begin{proof}
Let $u,v \in C(\overline{M \backslash K}) \cap C^1(M \backslash K)$ be bounded weak solutions of \eqref{exteriorProblem} such that $u \le v$ on $\partial K$, where $K \subset M$ is a compact set. Consider a ball $B_0=B(o,R_0)$ such that $K \subset \subset B_0$. Then $\Capa (B_0)=0$, since $M$ is $p$-parabolic. Therefore, Lemma \ref{SecondCharacterizationForP-parabolic} implies that there exist two sequences $R_1^k$ and $R_2^k$ such that $R_2^k > R_1^k \to +\infty$ and $\capa (R_1^k,R_2^k,o) \to 0$. Hence, using that $u$ and $v$ are bounded, we have that
$$ \max_{\mathpzc{A}_{R_1^k, R_2^k}} |u|^p  \capa (R_1^k,R_2^k) \quad {\rm and} \quad  \max_{\mathpzc{A}_{R_1^k, R_2^k}} |v|^p  \capa (R_1^k,R_2^k) $$
are bounded. Then, from Theorem \ref{GenericComparisonTheorem} we conclude that $u \le v$ in $M \backslash K$.
\end{proof}

Now we obtain the equivalence between (a) and (b) of Theorem \ref{MT} in the following result:

\begin{theorem}
Let $M$ be a complete noncompact Riemannian manifold and $p > 1$. The following are equivalents 
\\(i) $M$ is $p$-parabolic;
\\(ii) the comparison principle holds for the exterior
problem 
\begin{equation}
\Delta_{p} v = 0 \quad\mathrm{in} \quad M \backslash K,\label{p-harmonicInM-K}%
\end{equation}
for every compact set $K$ of $M$;
\\(iii) the comparison principle holds for the exterior
problem \eqref{p-harmonicInM-K} for some compact set $K_0=\overline{U}_0$, where $U_0\ne \emptyset$ is some bounded open set with $C^2$ boundary.
\label{p-harmonicResultForExteriorProb}
\end{theorem}

\begin{proof} 
$(iii) \Rightarrow (i):$ Suppose that $M$ is not $p$-parabolic. Then, $\Capa (E) > 0$ for some compact set $E$. Since, we are assuming $(iii)$, the comparison principle holds for some compact $K_0=\overline{U}_0$, where $U_0\ne \emptyset$ is open. Then, Corollary \ref{Dichotomy} implies that $\Capa (K_0) > 0$. Remind also that $\partial U_0$ is $C^2$. Therefore, from Lemma \ref{existenceOfNonConstantp-harmonic}, there exists a $p$-harmonic function $w$ such that $w=1$ on $\partial K_0$ and $0 < w < 1$ in $M \backslash K_0$. Let $u$ be defined by $u=1$ in $M$. Thus $u$ and $w$ are bounded $p$-harmonic functions, $u \le w$ on $\partial K_0$, but $w < 1 = u$ in $M \backslash K_0$. That is, the comparison principle does not hold for \eqref{p-harmonicInM-K} for $K_0=\overline{U}_0$, contradicting the hypothesis.
Hence $M$ is $p$-parabolic.

$(i) \Rightarrow (ii):$ Assuming that $M$ is $p$-parabolic, the comparison principle for the exterior domains with the $p$-laplacian operator is a consequence of Theorem \ref{generalResultForP-parabolicManifold}. 

$(ii) \Rightarrow (iii):$ Trivial.
\end{proof}

\

 From the previous results we obtain the following
property about the $p$-capacity, that corresponds the equivalence between (a) and (c) of Theorem \ref{MT}:

\begin{corollary}
Let $M$ be a complete noncompact Riemannian manifold and $p > 1$. There exist
two sequences $R_{1}^{k}$ and $R_{2}^{k}$ such that $R_{2}^{k} > R_{1}^{k}
\to+\infty$ and
\[
\mathrm{\; cap}_{p} \, (R_{1}^{k},R_{2}^{k},o) \to0
\]
if and only if there exist two sequences $\tilde{R}_{1}^{k}$ and $\tilde
{R}_{2}^{k} $ such that $\tilde{R}_{2}^{k} > \tilde{R}_{1}^{k} \to+\infty$
and
\[
\sup_{k} \mathrm{\; cap}_{p} \, (\tilde{R}_{1}^{k},\tilde{R}_{2}^{k},o) <
+\infty.
\]
Moreover, $M$ is $p$-parabolic if and only if some of these two conditions holds.
\end{corollary}

\begin{proof}

If $M$ is $p$-parabolic, then Lemma \ref{SecondCharacterizationForP-parabolic} guarantees that there exist two sequences $R_1^k$ and $R_2^k$ such that $R_2^k > R_1^k \to +\infty$ and $\capa (R_1^k,R_2^k,o) \to 0$. Therefore, it is trivial that $\sup_k \capa (R_1^k,R_2^k,o) < +\infty.$

Hence, if M is $p-$parabolic then the first condition holds which implies the second one.

Now suppose that
$$\sup_k \capa (\tilde{R}_1^k,\tilde{R}_2^k,o) < +\infty.$$
Hence, from Corollary \ref{genericComparisonResultForBoundedSol}, the comparison principle holds for the exterior problem \eqref{p-harmonicInM-K}. Then, Theorem \ref{p-harmonicResultForExteriorProb} implies that $M$ is $p$-parabolic, concluding the proof.
\end{proof}

\section{Comparison Principle under volume growth conditions}

In \cite{H2}, Holopainen proves that $M$ is $p$-parabolic assuming the following condition on the volume growth of geodesic balls: 
\begin{equation}
 \int^{+\infty} \left( \frac{R}{V(R)} \right)^{\frac{1}{p-1}} dR = +\infty \quad {\rm or} \quad \int^{+\infty} \left( \frac{1}{V'(R)} \right)^{\frac{1}{p-1}} dR = +\infty,
 \label{VolumeGrowthIlkkaCond}
\end{equation}
for $p > 1$, where $V(R) = {\rm Vol} ( B(o,R))$. Therefore, according to Theorem \ref{MT}, the $p$-comparison principle for exterior domains in $M$ holds if those conditions are satisfied.

In this section we follow a different kind of assumptions, using Theorem \ref{GenericComparisonTheorem} to obtain comparison principles. One advantage is that we can prove results even for solutions that are not bounded a priori and for manifolds that are not $p$-parabolic, provided there is some relation between the growth of the solution and the volume of the geodesic balls. 

We have to observe the following:

\begin{remark}\label {R}
For $R_{2} > R_{1} > 0$ and $o \in M$, there exists a function
\[
w_{0} \in\mathcal{F}_{B(o,R_{1}),B(o,R_{2})}%
\]
such that $|\nabla w_{0}(x)| \le2 / (R_{2}-R_{1})$ for any $x \in
\mathcal{A}_{R_{1},R_{2}} = B(o,R_{2}) \backslash\overline{B(o,R_{1})}$. For
instance, take $w_{0}$ as a mollification of
\[
w_{1}(x) = \left\{
\begin{array}
[c]{rcl}%
1 & if & dist(x,o) \le R_{1}\\[10pt]%
\displaystyle \frac{R_{2} - dist(x,o)}{R_{2}-R_{1}} & if & R_{1} < dist(x,o) <
R_{2}\\[10pt]%
0 & if & dist(x,o) \ge R_{2}.
\end{array}
\right.
\]
Then
\begin{equation}
\mathrm{\; cap}_{p} \, (R_{1},R_{2},o) \le\int_{\mathpzc{A}_{R_{1},R_{2}}}
|\nabla w_{0}|^{p} dx \le\left(  \frac{2}{R_{2}-R_{1}} \right) ^{p}
\text{Vol}(\mathcal{A}_{R_{1},R_{2}}).\label{capacityBoundedByVolume}%
\end{equation}
\label{limitationOfPhiByVolume}
\end{remark}

As a consequence of Remark \ref{R} and Theorem \ref{GenericComparisonTheorem}, we have the following result that can be applied for possibly unbounded solutions or non $p$-parabolic manifolds:

\begin{theorem}
Let $u,v \in C(\overline{M \backslash K}) \cap C^1(M \backslash K)$ be weak solutions of \eqref{exteriorProblem} in $M\backslash K$, where $A$ satisfies \eqref{AdeZeroeZero} - \eqref{UpperBoundForA} and $p >1$. Assume also that 
\begin{equation} \max_{\mathpzc{A}_{R_1^k, R_2^k}} |u|^p \;  \frac{\text{Vol}(\mathcal{A}_{R_1^k,R_2^k})}{(R_2^k-R_1^k)^p} \quad \quad {\rm and} \quad \quad \max_{\mathpzc{A}_{R_1^k, R_2^k}} |v|^p \;  \frac{\text{Vol}(\mathcal{A}_{R_1^k,R_2^k})}{(R_2^k-R_1^k)^p}
\label{boundedCapacityTimesSolutionHypotheses}
\end{equation}
are bounded sequences for some $(R_1^k)$ and $(R_2^k)$ such that $R_2^k > R_1^k \to \infty$.  
\\ If $u \le v$ on $\partial K$, then
$$ u \le v \quad \text{in} \quad M \backslash K.$$
\label{GenTheorForCompWithVolGrowth}
\end{theorem} 
\begin{corollary}
Assume the same hypotheses as in the previous theorem. Then the solutions $u$ and $v$ are bounded.
\label{BoundednessOfSolutions}
\end{corollary}

\begin{proof}
\underline{First Case:} The sequence of quotients
\begin{equation}
Q_k:=\frac{\text{Vol}(\mathcal{A}_{R_1^k,R_2^k})}{(R_2^k-R_1^k)^p} 
\label{boundedCapacityWITHOUTSolutionHypotheses}
\end{equation}
is also bounded for the same $(R_1^k)$ and $(R_2^k)$ given by Theorem \ref{GenTheorForCompWithVolGrowth}.
\\ Observe that the constant function $w=\displaystyle \max_{\partial K} v$ is a solution \eqref{exteriorProblem}.
Since the sequence in \eqref{boundedCapacityWITHOUTSolutionHypotheses} is bounded, then the sequence of \eqref{boundedCapacityTimesSolutionHypotheses} with $u$ replaced by $w$ is also bounded. Then we can apply the previous theorem to $w$ and $v$. Hence $v \le w = \displaystyle \max_{\partial K} v < +\infty$. Similarly, $v$ is bounded from below by $\displaystyle \min_{\partial K} v$. The argument for $u$ is the same.
\\ \underline{Second case:} The sequence of ratios $Q_k$ is not bounded. 
\\ Then there exists some subsequence $Q_{k_j}$ such that $Q_{k_j} \to +\infty$.
Hence, from the boundedness of the sequences given in \eqref{boundedCapacityTimesSolutionHypotheses}, we have
$$ \max_{\mathpzc{A}_{R_1^{k_j}, R_2^{k_j}}} |u|^p \to 0 .$$
In particular, this sequence is bounded by some constant $C > 0$. Therefore, $|u| \le C^{1/p}$ on the spheres $\partial B(o,R_1^{k_j})$ for any $j$. Since $u$ satisfies the maximum principle, we get
$$ |u| \le D:=\max\{ C^{1/p}, \max_{\partial K} |u| \} < +\infty \quad {\rm in} \quad B(o,R_1^{k_j}) \backslash K \quad \text{for any } j.$$
Using that $R_1^{k_j} \to +\infty$, we conclude that $|u| \le D$ in $M\backslash K$. The same holds for $v$.
\end{proof}

\begin{corollary}
Assume the same hypotheses as in Theorem \ref{GenTheorForCompWithVolGrowth}. If $u=v$ on $\partial
K$, then $u=v$ in $M \backslash K$. Moreover $u$ is bounded.
\end{corollary}

Observe that condition \eqref{boundedCapacityTimesSolutionHypotheses} holds, for example, if there exist $C > 0$, $q >0$ and some sequence $R_2^k \to +\infty$ such that 
$$
 \max_{B(o,R_2^k)} |u| \le C (R_2^k)^{(p-q)/p}  \quad , \quad \max_{B(o,R_2^k)} |v| \le C (R_2^k)^{(p-q)/p} 
$$ 
and
$$
\text{Vol}(B(o,R_2^k)) \le C (R_2^k)^q.
$$
In particular, condition \eqref{boundedCapacityTimesSolutionHypotheses} is satisfied if we assume that
\begin{equation}
 \max_{B(o,R)} |u|, \max_{B(o,R)} |v| \le C R^{(p-q)/p}  \quad {\rm and } \quad \text{Vol}(B(o,R)) \le C R^q, 
\label{potencilBoundOfSolAndPowerGrowth}
\end{equation}
for any $R \ge R_0$, where $R_0 >0$ is any fixed positive. Hence we have the following result:

\begin{corollary}
Let $u,v \in C(\overline{M \backslash K}) \cap C^1(M \backslash K)$ be weak solutions of \eqref{exteriorProblem} in $M\backslash K$, where $A$ satisfies \eqref{AdeZeroeZero} - \eqref{UpperBoundForA} and $p >1$. If \eqref{potencilBoundOfSolAndPowerGrowth} holds and $u \le v$ on $\partial K$, then $u \le v$ in $M\backslash K$. Moreover $u$ and $v$ are bounded. 
\end{corollary}
From this corollary, for $p > q$, we do not need to assume that the solutions are bounded to have a comparison principle. It is sufficient that they satisfy \eqref{potencilBoundOfSolAndPowerGrowth}. Anyway we conclude that they are bounded from Corollary \ref{BoundednessOfSolutions}.   On the other hand, if $p < q$, we cannot guarantee that $M$ is $p$-parabolic, since the quotient
given by \eqref{boundedCapacityWITHOUTSolutionHypotheses} may diverge to infinity. Still, we have some comparison principle provided we assume that $u$ and $v$ goes to zero at infinity with some speed.

\begin{remark}
If the sequence of $Q_k$, defined in \eqref{boundedCapacityWITHOUTSolutionHypotheses}, is bounded, then the sequence of capacities $\mathrm{\; cap}_{p} \, (R_{1}^k,R_{2}^k,o)$ is also bounded, according to \eqref{capacityBoundedByVolume}.
Then, from (c) of Theorem \ref{MT}, we conclude that $M$ is $p$-parabolic. Therefore, from Theorem \ref{generalResultForP-parabolicManifold}, it holds the comparison principle for bounded solutions.
\end{remark}

This implies the uniqueness of bounded solution for a Dirichlet problem in exterior domains. Morover,  from the fact that $v \equiv const.$ is a bounded solution of
\eqref{exteriorProblem}, we have the following extension of Liouville's result to exterior domains:

\begin{corollary}
Assume the same hypotheses as in Theorem \ref{GenTheorForCompWithVolGrowth} about $M$, $K$, $p$ and
$A$. Suppose that the quotient $Q_k$ of \eqref{boundedCapacityWITHOUTSolutionHypotheses} is bounded (or simply, $M$ is $p$-parabolic). If $u$ is a bounded weak solution of \eqref{exteriorProblem} and $u$ is
constant on $\partial K$, then $u$ is constant. More generally, given a continuous function $\phi$ on $\partial K$,  there exists at most one bounded solution of \eqref{exteriorProblem} such that $u =\phi$ on $\partial K$.
\end{corollary}

\section{Comparison principle for rotationally symmetric manifolds}

According to Milnor's lemma in \cite{M}, a complete rotationally symmetric $2$-dimensional Riemannian manifold with the metric $d \, s^2=d\, r^2 + f^2(r)d\, \theta$ is parabolic if and only if 
$$ \int_a^{\infty} \frac{1}{f(r)} dr = +\infty $$
for some $a > 0$. For higher dimension and $p >1$, a complete rotationally symmetric $n$-dimensional Riemannian manifold 
$M=\mathbb{R^{+}} \times S^{n-1}$ with respect
to a point $o \in M$ endowed with the metric
\[
d\, s^{2} = d\, r^{2} + f^{2}(r) d\, \omega^{2},
\]
where $r=dist(x,o)$, $d\, \omega^{2}$ is the standard metric of $S^{n-1}$ and
$f$ is a $C^{1}$ positive function in $(0,+\infty)$,
is $p$-parabolic if and only if
\begin{equation}
 \int_a^{\infty} f^{-\frac{n-1}{p-1}}(r) dr = +\infty \quad \text{for some } \; a > 0. \label{IntegralOfFisInfinity}
\end{equation}
Indeed, that the divergence of this integral implies the $p$-parabolicity is a consequence of the second integral condition in \eqref{VolumeGrowthIlkkaCond} (or Ilkka's condition) and the fact that the volume of the ball $B(o,r)$ is $V(r)=n \omega_n \displaystyle \int_0^r (f(s))^{n-1} ds$, where $\omega_n$ is the volume of the unit ball in $\mathbb{R}^n$. The converse is true since the convergence of this integral implies the existence of a nonconstant bounded $p$-superharmonic function: for instance, consider the function $\eta_{a,+\infty}$ defined by 
$\eta_{a,+\infty}(x) = 1$ if  ${\rm dist}(x,o) \le a$ and 
$$   \eta_{a,+\infty}(x) =  \frac{ \displaystyle \int_{{\rm dist}(x,o)}^{+\infty} f^{\frac{1-n}{p-1}}(s) \, ds }{
\displaystyle \int_{a}^{+\infty} f^{\frac{1-n}{p-1}}(s)\, ds } \quad {\rm if } \quad {\rm dist}(x,o) > a,   $$
that is $p$-harmonic outside the ball $B(o,a)$ as we will see later and, therefore, $p$-superharmonic in $M$. Hence, from Theorems \ref{p-harmonicResultForExteriorProb} and \ref{generalResultForP-parabolicManifold}, we conclude respectively the following results:
\begin{corollary}
Let $M=\mathbb{R^{+}} \times S^{n-1}$ be a complete rotationally symmetric $n$-dimensional Riemannian manifold 
with respect
to a point $o \in M$ endowed with the metric
\[
d\, s^{2} = d\, r^{2} + f^{2}(r) d\, \omega^{2} .
\]
Then the comparison principle holds for the exterior problem $\Delta_p v = 0$ in $M \backslash K$, where $K$ is any compact set, if and only if condition \eqref{IntegralOfFisInfinity} holds.
\end{corollary}  
\begin{corollary}
Let ($M$,$d s^2$) be as in the previous corollary and assume $A$ satisfies \eqref{AdeZeroeZero} -
\eqref{UpperBoundForA} for $p >1$. Suppose also that $u,v \in C(\overline{M \backslash K}) \cap C^{1}(M
\backslash K)$ are bounded weak solutions of \eqref{exteriorProblem} in
$M\backslash K$, where $K$ is a compact set of $M$, and that condition \eqref{IntegralOfFisInfinity} is satisfied.
If $u \le v$ on $\partial K$, then
\[
u \le v \quad\text{in} \quad M \backslash K.
\]
In particular, if $u$ is constant on $\partial K$, then $u$ is constant in
$M\backslash K$. \label{comparisonResultProvidedIntegralOfFisInfinity}
\end{corollary}

Now we show some comparison result that holds for non bounded solutions or hyperbolic manifolds. Taking the point $o$ as the reference, let $\eta=\eta_{R_{1},R_{2}}$ be the
function defined in Lemma \ref{etaDefinition} for $R_{2} > R_{1} > 0$. Since $\eta$ is $p$-harmonic in $\mathpzc{A}_{R_{1},R_{2}}$,  it is radially symmetric due to the symmetry of $M$ with respect to
$o$. Then $\eta$ is a function of $r$ in $\mathpzc{A}_{R_{1},R_{2}}$ and
satisfies the equation
\[
\frac{\Delta_p\eta}{|\nabla \eta|^{p-2}}= (p-1) \eta^{\prime\prime}(r) + (n-1)\frac{f^{\prime}(r)}{f(r)}
\eta^{\prime}(r) = 0 \quad\text{for} \quad r \in(R_{1},R_{2}),
\]
where the prime denotes the derivative with respect to $r$. Moreover,
$\eta(R_{1})=1$ and $\eta(R_{2})=0$. Hence
\[
\eta(r) = \frac{ \displaystyle \int_{r}^{R_{2}} f^{\frac{1-n}{p-1}}(s) \, ds }{
\displaystyle \int_{R_{1}}^{R_{2}} f^{\frac{1-n}{p-1}}(s)\, ds }.
\]
Then, using that the element of volume is $dx = n \omega_{n}f^{n-1}(r) dr$, we
have
\[
\int_{\mathpzc{A}_{R_{1},R_{2}}} |\nabla\eta|^{p} dx = n \omega_{n}
\int_{R_{1}}^{R_{2}} (\eta^{\prime })^{p} f^{n-1}(r) \, dr = \frac{n \omega_{n}}{
\left( \displaystyle \int_{R_{1}}^{R_{2}} f^{\frac{1-n}{p-1}}(s)\, ds \right)^{p-1}}.%
\]
Observe that the capacity $\capa (R_1,R_2) := \capa (B(o,R_1);B(o,R_1))$ is attained at $\eta_{R_1,R_2}$, since
it is harmonic and, therefore, minimizes the Dirichlet integral over $\mathpzc{A}_{R_{1}, R_{2}}$ in the set $$\{ v \in C^0(\overline{\mathpzc{A}_{R_{1}, R_{2}}}) \cap C^1(\mathpzc{A}_{R_{1}, R_{2}})  :  v=1 \; \text{on} \; \partial B(o,R_1) \; \text{and } \; v =0 \; \text{on} \;  \partial B(o,R_2) \}.$$
Therefore, from the last equation,
\[
\left( \max_{\mathpzc{A}_{R_{1}, R_{2}}}
|u| \right)^{p} \capa(R_1,R_2) = n \omega_{n}\left( \max_{\mathpzc{A}_{R_{1}, R_{2}}}
|u| \right)^{p} \left(  \int_{R_{1}}^{R_{2}} f^{-\frac{n-1}{p-1}}(s)\, ds \right) ^{-(p-1)}%
\]

\noindent From this expression, Theorem \ref{GenericComparisonTheorem} and Remark \ref{comparisonRemarkForPEqualTo2}, we have the following:

\begin{corollary}
Assume the same hypotheses as in the previous corollary
about $M$, $K$ and $A$. Let $u,v \in C(\overline{M \backslash K})
\cap C^{1}(M \backslash K)$ be weak solutions of \eqref{exteriorProblem} in
$M\backslash K$. Suppose also that
\begin{equation}
\frac{\max_{\mathpzc{A}_{R_{1}^{k}, R_{2}^{k}}} |u|^{p}}{ \left(\displaystyle \int
_{R_{1}^{k}}^{R_{2}^{k}} f^{-\frac{n-1}{p-1}}(s)\, ds \right)^{p-1}} \quad\text{and} \quad\frac
{\max_{\mathpzc{A}_{R_{1}^{k}, R_{2}^{k}}} |v|^{p}}{ \left(\displaystyle \int
_{R_{1}^{k}}^{R_{2}^{k}} f^{-\frac{n-1}{p-1}}(s)\, ds \right)^{p-1}} \quad\text{are bounded,}%
\label{conditionOnfToEstimatePhi2}
\end{equation}
where $R_{1}^{k}$ and $R_{2}^{k}$ are sequences such that $R_{2}^{k} >
R_{1}^{k} \to+\infty$. For $p=2$, this condition can be replaced by

\begin{equation}
\frac{\max_{\mathpzc{A}_{R_{1}^{k}, R_{2}^{k}}} |v- u|^{2}}{
\displaystyle \int_{R_{1}^{k}}^{R_{2}^{k}} f^{1-n}(s)\, ds } \to
0.\label{conditionOnfToEstimatePhi1}%
\end{equation}
If $u \le v$ on $\partial K$, then
\[
u \le v \quad\text{in} \quad M \backslash K.
\]
\label{RotationallySymmetricManifoldComparisonResult1}
\end{corollary}

As an application we present a result where the functions $u$
and $v$ can go to infinity, provided that their growth are bounded by some
specific function.

\begin{corollary}
Let $(M,d\,s^2)$ be as in the previous corollary,
where $0 < f(r) \le E_{1} r$ for some $E_{1} > 0$ and any large $r$. Suppose
that $A$ satisfies \eqref{AdeZeroeZero} - \eqref{UpperBoundForA} with $p = n$.
Let $K \subset M$ be a compact set and $u,v \in C(\overline{M \backslash K})
\cap C^{1}(M \backslash K)$ be weak solutions of \eqref{exteriorProblem} in
$M\backslash K$ such that
\begin{equation}
|u(r)| \le E_{2} \, (\ln r)^{\frac{n-1}{n}} \quad\text{and} \quad|v(r)| \le E_{2} \,
(\ln r)^{\frac{n-1}{n}}\label{rootSquareGrowthOfLogarithm}%
\end{equation}
for some $E_{2} > 0$ and any large $r$. If $u \le v$ on $\partial K$, then
\[
u \le v \quad\text{in} \quad M \backslash K.
\]

\end{corollary}

\begin{proof}
By hypothesis, there exists $R_0 > 0$ such that $ f(r) \le E_1 r$ for $r \ge R_0$.
Then,
$$ \int_{R_1^k}^{R_2^k} \frac{1}{f(s)} \, ds \ge \int_{R_1^k}^{R_2^k} \frac{1}{E_1 s} \, ds = \frac{1}{E_1} \ln \left( \frac{R_2^k}{ R_1^k} \right) \quad \text{for} \quad R_2^k > R_1^k > R_0.$$
Let $(R_k)$ be a sequence that goes to infinity. Take
$$R_1^k=R_k \quad \text{and} \quad R_2^k = (R_1^k)^2=(R_k)^2.$$
Hence, using the last inequality, we have
$$ \int_{R_1^k}^{R_2^k} \frac{1}{f(s)} \, ds \ge   \frac{1}{E_1} \ln R_k. $$
Therefore, using the growth hypothesis about $u$ and $v$, we have
$$ \frac{\max_{\mathpzc{A}_{R_1^k, R_2^k}} |u|^n}{\left( \displaystyle \int_{R_1^k}^{R_2^k} (f(s))^{-1}\, ds \right)^{n-1} } \le \frac{E_1^{n-1} E_2^n (\ln R_2^k)^{n-1}}{(\ln R_k)^{n-1}}= (2E_1)^{n-1}E_2^n $$
and a similar estimate for $v$.
Then, from Corollary \ref{RotationallySymmetricManifoldComparisonResult1}, we conclude the result.
\end{proof}

\begin{remark}
This corollary holds, for instance, for any complete noncompact rotationally symmetric Riemannian
manifold such that the sectional curvature is nonnegative, since $f(r) \le r$
in this case.

It can be applied also for some Hadamard manifolds provided the curvature goes
to zero sufficiently fast. 

\end{remark}

\newpage\noindent Ari Aiolfi \newline\noindent Universidade Federal de Santa
Maria\newline\noindent Brazil\newline\noindent ari.aiolfi@ufsm.br

\medskip

\noindent Leonardo Bonorino\newline\noindent Universidade Federal do Rio
Grande do Sul\newline\noindent Brazil\newline\noindent leonardo.bonorino@ufrgs.br

\medskip

\noindent Jaime Ripoll\newline\noindent Universidade Federal do Rio Grande do
Sul\newline\noindent Brazil\newline\noindent jaime.ripoll@ufrgs.br

\medskip

\noindent Marc Soret

\noindent Universit\'{e} de Tours \newline\noindent France

\noindent marc.soret@idpoisson.fr

\medskip

\noindent Marina Ville

\noindent Univ Paris Est Creteil, CNRS, LAMA, F-94010 Creteil \newline%
\noindent France

\noindent villemarina@yahoo.fr

\end{document}